\documentclass[11pt]{article}
\usepackage{epigamath}

\usepackage[notext]{kpfonts}
\usepackage{baskervald}

\setpapertype{A4}

\usepackage[english]{babel}

\usepackage[dvips]{graphicx}     

\removelength{1cm}

\title{Coincidence of two Swan conductors of abelian characters}
\titlemark{Coincidence of two Swan conductors of abelian characters}
\author{Kazuya Kato and Takeshi Saito}
\authoraddresses{
\authordata{Kazuya Kato}{\firstname{Kazuya} \lastname{Kato}\\
\institution{Department of Mathematics, University of Chicago, Chicago, IL, USA}\\
\email{kkato@math.uchicago.edu}} \\
\authordata{Takeshi Saito}{\firstname{Takeshi} \lastname{Saito}\\
\institution{School of Mathematical Sciences, University of Tokyo, Tokyo 153-8914, Japan}\\
\email{t-saito@ms.u-tokyo.ac.jp}}
}
\authormark{K. Kato and T. Saito}
\date{\vspace{-5ex}} 
\journal{\'Epijournal de G\'eom\'etrie Alg\'ebrique} 
\acceptation{Received by the Editors on April 24, 2019, and in final form on August 19, 2019. \\ Accepted on September 23, 2019.}

\newcommand{\articlereference}{Volume 3 (2019), Article Nr. 15}

\usepackage[all]{xy}

\allowdisplaybreaks

\newdimen\origiwspc
\origiwspc=\fontdimen2\font

\usepackage{float}

\numberwithin{equation}{section}

\makeatletter

\renewcommand{\p@equation}{\arabic{section}.\arabic{equation}\expandafter\@gobble}
\makeatother

\newcommand{\Sw}{{\rm Sw}}
\newcommand{\Br}{{\rm Br}}

\newcommand{\rsw}{{\rm rsw}}

\newcommand{\ab}{{\rm ab}}

\newcommand{\Gal}{{\rm Gal}}

\newcommand{\C}{{\mathbb C}}

\newcommand{\Q}{{\mathbb Q}}

\newcommand{\Z}{{\mathbb Z}}

\newcounter{para}
\numberwithin{para}{section}
\renewcommand{\thepara}{\arabic{section}.\arabic{para}}
\newenvironment{para}{\paragraph{\thepara.}\refstepcounter{para}}{}

\newtheorem{thm}[para]{Theorem}
\newtheorem{prop}[para]{Proposition}

\newtheorem{lem}[para]{Lemma}

\newenvironment{pf}{\proof[\proofname]}{\endproof}

\usepackage{amscd}

\begin{document}


\maketitle



\begin{prelims}

\vspace{-0.55cm}

\def\abstractname{Abstract}
\abstract{ There are two ways to define the Swan conductor of an abelian character of the absolute Galois group of a complete discrete valuation field.  We prove that these two Swan conductors coincide.}



\vspace{0.15cm}

\languagesection{Fran\c{c}ais}{%

\vspace{-0.05cm}
{\bf Titre. Co\"{\i}ncidence de deux conducteurs de Swan des caract\`eres ab\'eliens} \commentskip {\bf R\'esum\'e.} Il y a deux fa\c{c}ons de d\'efinir le conducteur de Swan d'un caract\`ere ab\'elien du groupe de Galois absolu d'un corps de valuation discr\`ete complet. Nous montrons que ces deux conducteurs de Swan co\"{\i}ncident.}

\end{prelims}


\newpage

\setcounter{tocdepth}{1} \tableofcontents

\section{Introduction} 

\begin{para} Let $K$ be a complete discrete valuation field, let $\bar K$ be a separable closure of $K$, and let us consider $\chi \colon \Gal(\bar K/K)\to \C^\times$ a homomorphism which factors through $\Gal(L/K)$ for a finite cyclic extension $L\subset \bar K$ of $K$. There are two definitions of the Swan conductor of $\chi$, one is defined by using the logarithmic upper ramification filtration on $\Gal(\bar K/K)$ defined geometrically \cite{AS},  and the other is defined by using  the filtrations on the unit groups of complete discrete valuation fields and  cup products in Galois cohomology \cite{Ksd}. 

We prove  that the
two Swan conductors coincide.

\end{para}

\begin{para}\label{I} We briefly review the two Swan conductors, which we denote in this paper by $\Sw(\chi)$ and $\Sw^{\ab}(\chi)$, respectively.

$\Sw(\chi)$  is defined as follows. There is a decreasing filtration $\Gal(\bar K/K)_{\log}^t$ indexed by $t\in \Q_{\geq 0}$ on $\Gal(\bar K/K)$ by closed normal subgroups called the logarithmic upper ramification groups.  $\Sw(\chi)$ is defined to be the smallest  $t\in \Q_{\geq 0}$ such that $\chi(\Gal(\bar K/K)_{\log}^s)=\{1\}$ for all $s>t$ (such $t$ exists). See  \cite[Theorem 3.16]{AS}
and Section \ref{sec:ram_groups}.

$\Sw^{\ab}(\chi)$  is defined as follows. Fix an injection ${\mathbb Q}/{\mathbb Z}
\to {\mathbb C}^\times$, say $r\mapsto \exp(2\pi\sqrt{-1} r)$, and identify $\chi$ with an element of $H^1(K, {\mathbb Q}/{\mathbb Z})=H^2(K,{\mathbb Z})$. Then the cup-product with $\chi$ defines a homomorphism $K^\times \to \Br(K)$, denoted by $a\mapsto \{\chi, a\}$,  where $\Br(K)=H^2(K,{\mathbb G}_m)$ is the Brauer group of $K$ \cite[Chapitre X.4]{CL}.  $\Sw^{\ab}(\chi)$ is defined to be the smallest integer $n\geq 0$ such that $\{\chi, 1+{\mathfrak m}_K^n{\mathfrak m}_{K'}\}=0$ in $\Br(K')$  for any extension $K\to K'$ of complete discrete valuation fields (such $n$ exists; here ${\mathfrak m}_*$ denote the maximal ideals). See \cite{KKs}.
 \end{para}

\begin{thm}\label{th01} $\Sw(\chi)=\Sw^{\ab}(\chi)$. 

\end{thm}

\begin{para}\label{rsw}
We will also prove the coincidence of refined Swan conductors. Let $r=\Sw(\chi)=\Sw^{\ab}(\chi)$ and assume $r>0$. Then we have non-zero elements 
$$\rsw(\chi) \in \bar F\otimes_F {\mathfrak m}_K^{-r}/{\mathfrak m}_K^{-r+1} \otimes_{{\cal O}_K} \Omega^1_{{\cal O}_K}(\log),$$
$$\rsw^{\ab}(\chi) \in {\mathfrak m}_K^{-r}/{\mathfrak m}_K^{-r+1} \otimes_{{\cal O}_K} \Omega^1_{{\cal O}_K}(\log)$$
called the (logarithmic) refined Swan conductors. $\rsw(\chi)$  is defined in \cite{Sa} as a refined version of $\Sw(\chi)$ and $\rsw^{\ab}(\chi)$ is defined in \cite{KKs} as a refined version of $\Sw^{\ab}(\chi)$.  
The definition of $\rsw(\chi)$ is recalled in (\ref{eqrsw}).
\end{para}

\begin{thm}\label{th02}
$\rsw(\chi)=\rsw^{\ab}(\chi).$
\end{thm}

\begin{para}\label{cases} Theorem \ref{th01} was known, for example,  in the following cases. (The cases (i) and (ii) follow from Section 6.1 of \cite{AS} and from \cite{Ksd}. The case (iii) is shown in \cite[Corollary 9.12]{AS2}.)
The case (ii) is proved by comparing Propositions \ref{prmain}
and \ref{Swe=1}.
In the cases (i) and (ii),
the ring ${\cal O}_L$ is
generated by a single element over
${\cal O}_K$ and such an extension is also studied in \cite{Sp}.

\begin{enumerate}
\item[(i)] (The classical case.) The case where the residue field of $K$ is perfect.
\item[(ii)] The case where the residue field $F$ of $K$ is of characteristic $p$ such that $[F:F^p]=p$ and
 $\chi$ factors through $\Gal(L/K)$ for a finite cyclic extension $L/K$ whose ramification index $e(L/K)$ is one. 
\item[(iii)] The case where $K$ is of positive characteristic. 
\end{enumerate}

\noindent
In the positive characteristic case,
the results corresponding to Theorems \ref{th01} and \ref{th02}
in the non-logarithmic case are proved in \cite{AS2}
and \cite{Ya}.

\end{para}

\begin{para}\label{K'} Our method to prove Theorem \ref{th01} is to reduce it to the above case (ii) (not to the classical case (i)). In Theorem \ref{th2}, we prove
that for a finite cyclic extension $L/K$ such that $\chi$ factors through $\Gal(L/K)$, we can find  an extension of  complete discrete valuation fields $K\to K'$ such that $\Sw(\chi_{K'})= e(K'/K)\Sw(\chi)$, $\Sw^{\ab}(\chi_{K'})= e(K'/K)\Sw^{\ab}(\chi)$,  the residue field $F'$ of $K'$ satisfies $[F': (F')^p]=p$, and $e(LK'/K')=1$. 
The refined Swan conductors play important roles to find the field $K'$ above. 

Theorem \ref{th02} is proved also by the reduction to the case (ii). 
The authors would like to thank an anonymous referee
for pointing out that almost the same result
as the key step Proposition \ref{prmain}
is proved in \cite[Theorem 5.9]{Hu}.
\end{para}

\medskip

One of the authors (K.~K.) is partially supported by
NSF Award 1601861
and (T.~S.) is partially supported by JSPS 
Grant-in-Aid for Scientific Research (A) 26247002. 

\section{On the theorem of Epp}\label{s2}

The following theorem is not explicitly written in the paper \cite{Ep} of Epp, but the arguments there (with a correction in \cite{Ku} of an error in \cite{Ep}) actually prove this. 

\begin{thm}\label{th1}Let $K$ be a complete discrete valuation field whose residue field $F$ is of characteristic $p>0$, and let $L$ be a finite Galois extension of $K$. 
Then there exists a finite extension $K'$ of $K$ satisfying the following conditions {\rm (i)} and {\rm (ii)}.
\begin{enumerate}
\item[\rm (i)] $e(LK'/K')=1$. 
\item[\rm (ii)] The residue field of $K'$ is
a separable extension of that of $K$. 
\end{enumerate}
\end{thm}

In Theorem \ref{th1},
we may take $K'$ separable over $K$,
although we will not use this fact.
To see this,
it suffices to modify the construction 
of $K'=K(\pi')$ in the proof of
the case where $K$ is of characteristic $p>0$ and $T$ is not empty in \ref{pos}.

We use the following lemmas \ref{FE}, \ref{le1}  and \ref{le2} for the proof of Theorem \ref{th1}.

For a discrete valuation field $K$, let $\text{ord}_K$ be the normalized additive  valuation of $K$. In the case the residue field of $K$ is of characteristic $p>0$, let $e_K= \text{ord}_K(p)$. (So, $e_K=\infty$ if $K$ is of characteristic $p$.)
%
%
%
%

\begin{lem}\label{FE} Let $K$ be a complete discrete valuation field whose residue field $F$ is of characteristic $p>0$.
Let $k
=\bigcap_{r\geq 0}F^{p^r}$ be
the largest perfect subfield of $F$
and let $W(k)\to {\cal O}_K$ be
the canonical morphism from
the ring of Witt vectors.
Then the subring $\bigcap_{r\geq 0} ({\cal O}_K/p{\cal O}_K)^{p^r}\subset {\cal O}_K/p{\cal O}_K$ equals the image of $k\to {\cal O}_K/p{\cal O}_K$ \emph{(}in the case $K$ is of characteristic $p$, this means that $\bigcap_{r\geq 0} ({\cal O}_K)^{p^r}=k$\emph{)}.

\end{lem}

\begin{pf}
Let $A=\bigcap_{r\geq 0} ({\cal O}_K/p{\cal O}_K)^{p^r}$ denote the subring.
Then, $A\subset {\cal O}_K/p{\cal O}_K$ contains the image of $k$ and
the image of $A$
by ${\cal O}_K/p{\cal O}_K\to F$
is a subring of $k$.
Hence, the assertion follows from 
$A\cap ({\mathfrak m}_K/p{\cal O}_K)=0$.
\end{pf}

We do not give proofs of the following lemmas \ref{le1} and \ref{le2} which are straightforward.

\begin{lem}\label{le1}  Let $K$ be a complete discrete valuation field of characteristic $p>0$  and let  $F$ be its residue field. Consider the Artin-Schreier extension  $L=K(\alpha)$, $\alpha^p-\alpha=f\in K$. Let $\pi$ be a prime element of $K$. Let $E$ be the residue field of $L$.
\begin{enumerate}
\item[\rm (1)] If $f\in {\cal O}_K$, the extension $L/K$ is unramified,
possibly trivial. 
\item[\rm (2)] Assume that $-{\rm ord}_Kf
=n\geq 1$ is not divisible by $p$. Then $e(L/K)=p$ and $E=F$. 
\item[\rm (3)]
Assume that $f\in  u\pi^{-mp}+ \pi^{-mp+1}{\cal O}_K$
for some integer $m\geq 1$ and for some $u\in {\cal O}_K$ whose residue class $\bar u$ does not belong to $F^p$. Then $E= F(\bar u^{1/p})$ and $e(L/K)=1$.   
\end{enumerate}
\end{lem}

\begin{lem}\label{le2}  Let $K$ be a complete discrete valuation field of mixed characteristic $(0,p)$. Let  $F$ be its residue field. Assume that $K$ contains a primitive $p$-th root $\zeta_p$ of $1$. Consider the Kummer extension  $L=K(\alpha)$, $\alpha^p=a\in K^\times$. Let $\pi$ be a prime element of $K$.  Let $E$ be the residue field of $L$.
\begin{enumerate}
\item[\rm (1)] If $a\in 1+(\zeta_p-1)^p{\cal O}_K$, the extension $L/K$ is unramified, possibly trivial. 
\item[\rm (2)] Assume that $\text{ord}_K(a)$ is not divisible by $p$. Then $e(L/K)=p$ and $E=F$.
\item[\rm (3)] Assume that $a\in ({\cal O}_K)^\times$ and that the residue class $\bar a$ of $a$ is not contained in $F^p$. Then $E=F(\bar a^{1/p})$ and $e(L/K)=1$. 
\item[\rm (4)] Assume that $a\in  (1+\pi^n u)(1+\pi^{n+1}{\cal O}_K)$ for some integer $n$ not divisible by $p$ such that $1\leq n< e_Kp/(p-1)$ and for some $u\in ({\cal O}_K)^\times$. Then $e(L/K)=p$ and $E=F$.  
\item[\rm (5)] Assume that $a\in  (1+\pi^{mp} u)(1+\pi^{n+1}{\cal O}_K)$ for some integer $m$ such that $1\leq n=mp< e_Kp/(p-1)$ and for some $u\in {\cal O}_K$ whose residue class $\bar u$ does not belong to $F^p$. Then $E= F(\bar u^{1/p})$ and $e(L/K)=1$.  
\end{enumerate}
\end{lem}

\begin{para}\label{start} We start the proof of Theorem \ref{th1}.  

First, we reduce the theorem to the case (*) below.
Let $K_1\subset L$ be the maximum tamely ramified
extension of $K$.
Then, since $LK_1=L$ and the residue
field of $K_1$ is a separable extension
of that of $K$, 
we may assume that ${\rm Gal}(L/K)$
equals the inertia subgroup $I$ and
is a $p$-group.

We prove the reduction by induction on
the order of $I$.
We may assume that $L\neq K$.
Then, since ${\rm Gal}(L/K)$
is nilpotent, 
there exists a subextension $L'\subset L$ 
such that $L'$ is a Galois extension over $K$
and that $L$ is a cyclic extension of $L'$ of degree $p$.
By induction hypothesis,
there exists a finite extension $K'_1$ of
$K$ such that
$e(L'K'_1/K'_1)=1$ and satisfying (ii).
If $e(LK'_1/L'K'_1)=1$, there is nothing to prove.
Otherwise,
for the maximum unramified extension
$K'_2$ of $K'_1$ inside $M'_1=L'K'_1$,
the extensions
$K'_2\subset M'_1\subset LK'_1$ satisfies 
the condition (*).

\begin{enumerate}
\item[(*)] There exists a field $M$ such that $K\subset M\subset L$, $e(M/K)=1$
and that $L$ is a cyclic extension of $M$ of degree $p$ and $e(L/M)=p$. 
The residue field $E$ of $M$
is a purely inseparable extension 
of the residue field $F$ of $K$.
\end{enumerate}
\end{para}

\begin{para}\label{pos}
We prove Theorem \ref{th1} in the case $K$ is of characteristic $p$.  Let $M$ be as in (*) in \ref{start}. We may assume $M=E((\pi))$ with $\pi$  a prime element of $K$. We can write $L=M(\alpha)$ where $\alpha^p-\alpha=f=f_I+f_U$ with $f_I, f_U\in M$ such that: $f_I=\sum_{n\in I} a_n \pi^{-n}$ where $I$ is a finite subset of $\Z_{>0}$ and  $a_n\in E^\times$, and $f_U\in {\cal O}_M$.  By Lemma \ref{le1} (1) applied to the extension $L/M$, $I$ is not empty because $L/M$ is not unramified. .

In the following, we use the fact that for $u,v\in M$ such that $u\equiv v \bmod \{w^p-w\;|\; w\in M\}$,  the extension $M(\beta)$, $\beta^p-\beta=u$, of $M$ is the same as that given by $v$. 
If $n\in I$ is divisible by $p$ and $a_n\in E^p$, we have $a_n\pi^{-n}\equiv a_n^{1/p}\pi^{-n/p}\bmod \{w^p-w\;|\; w\in M\}$ and hence we can replace $a_n\pi^{-n}$ by $a_n^{1/p}\pi^{-n/p}$. Hence we may (and do) assume that if $n\in I$ is divisible by $p$, then $a_n\notin E^p$. 

Let $S$ be the subset of $I$ consisting of all $n\in I$ such that $a_n\in \bigcap_{r\geq 0} E^{p^r}
=\bigcap_{r\geq 0} F^{p^r}=k$, and let $T=I\smallsetminus S$.
Note that if $n\in S$, then $n$ is not divisible by $p$. 
By Lemma \ref{FE}, we have $a_n\in 
k\subset {\cal O}_K$ for $n\in S$. Hence $f_S\in K$. 

Assume first $T$ is empty. Then $f_I=f_S\in K$. For  $K'=K(\alpha_S)$ with $\alpha_S^p-\alpha_S=f_S$, the residue field of $K'$ coincides with $F$ by Lemma \ref{le1} (2) applied to $K'/K$, and the extension $LK'/MK'$ is unramified by Lemma \ref{le1} (1) applied to $LK'/MK'$. 

Assume that $T$ is not empty. For $n\in T$, write $a_n=b_n^{p^{r(n)}}$ where $b_n\in E$, $r(n)\geq 0$, and $b_n$ is not a $p$-th power in $E$. Take an integer $m$ such that $m>r(n)$ for any $n\in T$. For $n\in S$, write $a_n=b_n^{p^m}$ ($b_n\in
k\subset {\cal O}_K$).

Let $K'= K(\pi')$ where $\pi'$ is a $p^m$-th root of $\pi$ and let $M'=MK'$, $L'=LK'$. Then 
$$f_I\equiv   f_S+f_T \bmod \{w^p-w\;|\; w\in M'\},$$ $$f_S= \sum_{n\in S} b_n(\pi')^{-n}, \quad f_T=\sum_{n\in T} b_n(\pi')^{-np^{m-r(n)}}.$$ 
Note that $f_S\in k((\pi'))\subset K'$ by Lemma \ref{FE}. Let $n_S:= \max(S)$ and $n_T:= \max\{np^{m-r(n)}\;|\;n \in T\}$. 
If $S$ is empty, we set $n_S=1$ so that we have $n_S<n_T$. Since $n_S$ is not divisible by $p$ and $n_T$ is divisible by $p$, we have $n_S\neq n_T$. For the proof of Theorem \ref{th1}, it is sufficient to prove the following Claim 1 and Claim 2.

\medskip

{\bf Claim 1.} If $n_S<n_T$, then
$e(L'/M')=1$.

\medskip

{\bf Claim 2.} If $n_S>n_T$, let $K''=K'(\beta)$ where $\beta^p- \beta = f_S$ and let $M''= MK''$, $L''=LK''$. Then the residue field of $K''$ coincides with that of $K$ and $e(L''/M'')=1$. 

\medskip

We first prove 

\medskip

{\bf Claim 3.} There is a unique $n\in T$ such that $np^{m-r(n)}=n_T$. 
\medskip

We prove Claim 3. If $n, n'\in T$, $n>n'$  and $np^{m-r(n)}= n'p^{m-r(n')}$, then by $n= n'p^{r(n)-r(n')}>n'$, we have $p|n$. Hence $a_n\notin E^p$ and $r(n)=0$. 
This contradicts to $r(n)>r(n')$.

Claim 1 follows from Claim 3 and Lemma \ref{le1} (3) applied to the extension $L'/M'$. 

We prove Claim 2. We have $e(K''/K')=p$ by Lemma \ref{le1} (2) applied to $K''/K'$.  If $\tau$ denotes a prime element of $K''$, the residue class of the unit $\tau^p(\pi')^{-1}$ is a $p$-th power. Claim 2 follows from this and Claim 3, and from Lemma \ref{le1} (3) applied to the extension $L''/M''$.

\end{para}

\begin{para}\label{mixed} We prove Theorem \ref{th1} in the case 
$K$ is of mixed characteristic $(0, p)$. We may assume that $K$ contains a primitive $p$-th root $\zeta_p$ of $1$. Note that $\text{ord}_K(\zeta_p-1)=e_K/(p-1)$.  Let $M$ be as in (*) in \ref{start}. We have $L=M(\alpha)$, $\alpha^p=a$
for $a\in M^\times$. 

The proof consists of two steps. In Step 1, we show that we may assume $a\in 1+p{\cal O}_M$. In Step 2, we give the proof assuming $a\in 1+p{\cal O}_M$.

Let $E$ be the residue field of $M$ and take a ring homomorphism $E\to {\cal O}_M/p{\cal O}_M$ such that the induced map $E\to {\cal O}_M/{\mathfrak m}_M=E$ is the identity map,
and its lifting $\iota \colon E\to {\cal O}_M$. Let $\pi$ be a prime element of $K$.

Step 1. Write 
$a\equiv  c \prod_{n\in T} c_n\bmod 1+p{\cal O}_M$ where $T$ is a subset of $\{0,\dots, e_M-1\}$ and $c_n$ ($n\in T$) and $c$ are elements of $M^\times$ of the following form. If $0\in T$, $c_0=\iota(b)$ for some $b\in E$ such that $b\notin E^p$. If $n\in T$ and $n\geq 1$, $c_n=1+\pi^n\iota(b)$ for some $b\in E$ such that $b\notin E^p$. The first term
$c$ is a product of a power of $\pi$ and elements of the form $1+\pi^m \iota(b)$ with $b\in E^p$
for some integer $m\geqq 0$. 

Let $K'=K(\pi^{1/p})$, $M'=MK'$, $L'=LK'$. Then we have $c\in ((M')^\times)^p(1+p{\cal O}_{M'})$ 
since the map $x\mapsto x^p$ on ${\cal O}_{M'}/p{\cal O}_{M'}$ is a ring homomorphism. Hence
if $T\neq \varnothing$,  we have $e(L'/M')=1$  by Lemma \ref{le2} (3) and (5) applied to the extension $L'/M'$.
If $T=\varnothing$,  we also have $L'=M'(\beta)$ with $\beta^p\in 1+p{\cal O}_{M'}$.  Thus, the assertion is reduced to the case where $L=M(\alpha)$, $\alpha^p=a$
for $a\in 1+p{\cal O}_M$.

Step 2. Assume $L=M(\alpha)$, $\alpha^p=a\in 1+p{\cal O}_M$. We have an isomorphism  $$ (\zeta_p-1)^{-1}{\cal O}_M/{\cal O}_M\to (1+p{\cal O}_M)/((1+(\zeta_p-1)^p{\cal O}_M)\;;\;x\mapsto 1+(\zeta_p-1)^px$$ (from the additive group to the multiplicative group). This isomorphism maps $x^p-x$ for $x\in (\zeta_p-1)^{-1}{\cal O}_M$ such that $x^p\in (\zeta_p-1)^{-1}{\cal O}_M$ 
to a $p$-th power because
$$(1+(\zeta_p-1)x)^p \equiv 1+(\zeta_p-1)^p(x^p-x) \bmod 1+(\zeta_p-1)^p{\cal O}_M.$$  Hence we have a situation similar to the theory of Artin-Schreier extension, and the rest of the proof, which is given below, is similar to  the proof of the case where $K$ is of characteristic $p$ in \ref{pos}.

We have  $a= a_Ia_U$ with $a_I=1+\sum_{n\in I} ((\zeta_p-1)^p\pi^{-n})\iota(a_n)$
where $I$ is a subset of $\{n\in \Z\;|\; 1\leq n \leq e_M/(p-1)\}$  and $a_n\in E^\times$ and $a_U\in 1+(\zeta_p-1)^p{\cal O}_M$. Note that we have $(\zeta_p-1)^p\pi^{-n}\in p{\cal O}_M$ for $n\in I$. By Lemma \ref{le2} (1) applied to the extension $L/M$, $I$ is not empty. We may assume that if $n\in I$ and $n$ is divisible by $p$, then $a_n$ is not a $p$-th power in $E$. Let $k= \bigcap_{r\geq 0} F^{p^r}= \bigcap_{r\geq 0} E^{p^r}$,
$S=\{n\in I\;|\; a_n \in k\}$ and let $T=I\smallsetminus S$.  

If $T$ is empty, by Lemma \ref{FE}, we have $a_I \equiv c \bmod (M^\times)^p$ where $c= 1+ \sum_{n\in S} ((\zeta_p-1)^p\pi^{-n})[a_n]$ for 
the Teichm\"uller lifting $[a_n]\in W(k)^\times \subset {\cal O}_K^\times$. Let  $K'=K(c^{1/p})$. Then  the residue field of $K'$ is $F$  by Lemma \ref{le2} (4) applied to the extension $K'/K$
and the residue field of $K'$ is the same as that of $K$, and the extension $LK'/MK'$ is unramified by Lemma \ref{le2} (1) applied to $LK'/MK'$. 

Assume now that $T$ is not empty. For $n\in T$, define
$b_n\in E\smallsetminus E^p$
and $r(n)\geq 0$ as in  \ref{pos}.
Further take an integer $m$ such that $m>r(n)$ for any $n\in T$
and $b_n\in k$ for $n\in S$ as in  \ref{pos}.

Let $K'= K(\pi')$ where $\pi'$ is a $p^m$-th root of $\pi$ and let $M'=MK'$, $L'=LK'$ Then by Lemma \ref{FE}, 
$$a_I \equiv a_Sa_T\bmod (1+(\zeta_p-1)^p{\cal O}_{M'}),$$ $$K'  \ni a_S= 1+\sum_{n\in S} ((\zeta_p-1)^p(\pi')^{-n})[b_n], \quad a_T=1+\sum_{n\in T} ((\zeta_p-1)^p(\pi')^{-np^{m-r(n)}})\iota(b_n)$$ 
where $[b_n]
\in W(k)^\times \subset {\cal O}_K^\times$ for $n\in S$ is the Teichm\"uller lifting of $b_n$. Let $n_T:= \max\{np^{m-r(n)}\;|\;n \in T\}$ and $n_S:= \max(S)$. Since $n_S$ is not divisible by $p$ and $n_T$ is divisible by $p$, we have $n_S\neq n_T$. For the proof of Theorem \ref{th1}, it is sufficient to prove the following Claim 1 and Claim 2.

\medskip

{\bf Claim 1.} If $n_S<n_T$, then
$e(L'/M')=1$. 

\medskip

{\bf Claim 2.} If $n_S>n_T$, let $K''=K'(a_S^{1/p})$. Then the residue field of $K''$ coincides with that of $K$ and $e(L''/M'')=1$ where $L''=LK''$ and $M''=MK''$.

\medskip

We first prove 

\medskip

{\bf Claim 3.} There is a unique $n\in T$ such that $np^{m-r(n)}=n_T$. 
\medskip

The proof of Claim 3 is similar to  that of Claim 3 in \ref{pos}. Claim 1 follows from Claim 3 and Lemma \ref{le2} (5) applied to the extension $L'/M'$. 
We prove Claim 2. The residue field of $K''$ is $F$ by Lemma \ref{le2} (4) applied to the extension $K''/K'$ and we have $e(K''/K')=p$.  If $\tau$ denotes a prime element of $K''$, the residue class of the unit $\tau^p(\pi')^{-1}$ is a $p$-th power. Claim 2 follows from this and  Claim 3 and from Lemma \ref{le2} (5) applied to the extension $L''/M''$. 

\end{para}

\section{Some extensions of complete discrete valuation fields}

\begin{thm}\label{th2} Let $K$ be a complete discrete valuation field whose residue field $F$  is of characteristic $p>0$. Let $L/K$ be a finite Galois extension. Then there is an extension $K'/K$ of complete discrete valuation fields satisfying the following conditions {\rm (i)}--{\rm (iii)}. Let $F'$ be the residue field of $K'$.
\begin{enumerate}
\item[\rm (i)] $e(LK'/K')=1$. 
\item[\rm (ii)] $[F':(F')^p]=p$. 
\item[\rm (iii)] The map $\Omega^1_F(\log) \to \Omega^1_{F'}(\log)$ is injective \emph{(}here and in the following $\Omega^1_F(\log)=F\otimes_{{\cal O}_K} \Omega^1_{{\cal O}_K}(\log)$\emph{)}. 
\end{enumerate}
If $F$ is finitely generated over a perfect subfield $k$, we can replace {\rm (ii)} by the following  stronger condition {\rm (ii)'}.
\begin{enumerate}
\item[\rm (ii)'] There is a perfect subfield $k'$ of $F'$ such that  $F'$ is finitely generated and of transcendence degree $1$ over $k'$.
\end{enumerate}
\end{thm}

We will deduce Theorem \ref{th2}  from Theorem \ref{th1} and the following Propositions \ref{p1} and \ref{p2}. 

\begin{prop}\label{p1} Let $K$ be a complete discrete valuation field whose residue field $F$ is of characteristic $p>0$. 
Let $\pi$ be a prime element of $K$, let ${\cal O}_{K'}$ be the completion of the discrete valuation ring which is the local ring of ${\cal O}_K[T, U^{\pm 1}]/(UT^p-\pi)$ at the prime ideal generated by $T$,  let $K'$ be the field of fractions of ${\cal O}_{K'}$, and let $F'$ be the residue field of $K'$. Then we have:
\begin{enumerate}
\item[\rm (i)] The map $\Omega^1_F(\log) \to \Omega^1_{F'}(\log)$ is injective. 
\item[\rm (ii)] The image of this map is contained in $\Omega^1_{F'}$. 
\end{enumerate}
\end{prop}

\begin{pf} Straightforward. 
\end{pf}

\begin{prop}\label{p2} Let $K$ be a complete discrete valuation field whose residue field $F$ is of characteristic $p>0$. Then there is an extension $K\to K'$ of  complete discrete valuation fields  satisfying the following conditions {\rm (i)}--{\rm (iii)}. Let $F'$ be the residue field of $K'$. 
\begin{enumerate}
\item[\rm (i)] $e(K'/K)=1$.
\item[\rm (ii)] $[F':(F')^p]=p$.
\item[\rm (iii)] The map $\Omega^1_F\to \Omega^1_{F'}$ is injective. 
\end{enumerate}
If $F$ is finitely generated over a perfect field $k$, we can replace {\rm (ii)} by the following  stronger condition {\rm (ii)'}.
\begin{enumerate}
\item[\rm (ii)'] There is a perfect subfield $k'$ of $F'$ such that  $F'$ is finitely generated and of transcendence degree $1$ over $k'$.
\end{enumerate}
\end{prop}

\begin{pf} Let $(b_i)_{i\in I}$ be a lifting of a $p$-base of $F$ to ${\cal O}_K$. Let 
$A_0={\cal O}_K[T_i, U\;;\; i \in I]$ 
be the polynomial ring
and set $S_i=b_i-UT_i\in A_0$.
The residue field $F_0$ of
$A_0$ at the prime ideal $\mathfrak p_0$
generated by ${\mathfrak m}_K$
is $F_0=F(T_i, U\;;\;i\in I)$
and $(S_i, T_i,U\;;\;i\in I)$ is a
$p$-base.
For integers $n\geq0$,
writing $T_i=T_{i,0}$ and $S_i=S_{i,0}$,
define 
$A_{n+1}
=
A_n[T_{i,n+1}, S_{i, n+1}\;;\; i \in I]/
(T_{i,n+1}^p-T_{i,n}, S_{i, n+1}^p-S_{i,n}\;;\; i \in I)$ inductively
and $A=\varinjlim_nA_n$.
Then $A$ is an integral domain, the ideal $\mathfrak p$ of $A$ generated by ${\mathfrak m}_K$ is a prime ideal
and the local ring $A_{\mathfrak p}$ is a discrete valuation ring.
Hence, the residue field $F'$ of
$A$ at $\mathfrak p$ is the extension of  $F_0$ obtained by adding $T_i^{1/p^n}$ and $S_i^{1/p^n}$ for all $i,n$. Let ${\cal O}_{K'}$ be the completion of the discrete valuation
ring $A_{\mathfrak p}$ and let $K'$ be the field of fractions of ${\cal O}_{K'}$.

Then $K'$ satisfies the conditions (i)--(iii). For (i) and (ii), this is clear. We prove that (iii) is satisfied. The $F'$-vector space $\Omega^1_{F'}$ is one-dimensional with base $dU$. The $F$-vector space $\Omega^1_F$ is with base $db_i$ ($i\in I$). In $\Omega^1_{F'}$, we have  $db_i= T_idU$. Since $T_i$ ($i\in I$) are linearly independent over $F$, we have the injectivity.

Assume that $F$ is finitely generated over a perfect field $k$. Then $I$ is finite and $F$ is a finite extension of $k(b_i\;;\;i\in I)$.  Let $k'\subset F'$ be the extension of the rational function field $k(S_i,T_i\;;\;i\in I)\subset F_0=F(T_i, U\;;\;i\in I)$ given by $k':=\bigcup_{n\geq 0} k(S_i^{1/p^n}, T_i^{1/p^n})$. Then $k'$ is perfect and $F'$ is a finite extension of $k'(U)$. 
\end{pf}

\begin{para}\label{p5} We prove Theorem \ref{th2}.

Let $K_1/K$ be the extension in Proposition \ref{p1}. By taking $K_1$ as $K$ in Proposition \ref{p2}, let $K_2/K_1$ be the extension $K'/K$ of Proposition \ref{p2}. Let $K_3/K_2$ be the maximal unramified subextension of $LK_2/K_2$. Then the extension $K_3/K$ satisfies (ii) and (iii) of \ref{th2} and (ii) of Proposition \ref{p1}. 

By applying Theorem \ref{th1} to $LK_3/K_3$, we obtain a finite extension $K'/K_3$ such that $e(LK'/K')=1$
and the residue field of $K'$
is a separable extension of that of $K_3$.
 The extension  $K'/ K$ has the desired properties. If $F$ is finitely generated over a prime field $k$, the condition (ii)' is satisfied. 
 
\end{para}

\section{Review and complements on ramification groups}\label{sec:ram_groups}

We briefly recall the definition
and basic properties of ramification groups.
For more detail, we refer to
\cite{AS, AS2, Jus, Sa, Sa2}.
We introduce the refined logarithmic conductor
for a finite Galois extension
of a henselian valuation field
in (\ref{eqom}),
as a generalization of
the refined Swan conductor
of an abelian character in the case
where the extension is cyclic.
We recall the definition of the Swan conductor
of an abelian character at the end of Section \ref{s43}.
In the case where the residue field
is a function field of one variable
over a perfect field and
the ramification index of the extension is one,
we compute explicitly
the refined logarithmic conductor
in Proposition \ref{prmain}
using Lemma \ref{lmmono}.

\begin{para}
Let $K$ be a henselian discrete valuation
field and $F$ be the residue field
of the valuation ring ${\cal O}_K$.
Let $\bar K$
be a separable closure of $K$
and $G_K={\rm Gal}(\bar K/K)$ be the
absolute Galois group.
The residue field $\bar F$ of $\bar K$
is an algebraic closure of $F$.

Let $L$ be a finite \'etale $K$-algebra
and $r>0$ be a rational number.
Let ${\rm Spec}\, {\cal O}_L\to Q$
be a closed immersion
to a smooth scheme $Q$ 
over ${\cal O}_K$.
Let $K'\subset \bar K$ be a finite extension of $K$
of ramification index $e$
such that $er$ is an integer.
Then, we define
a dilatation $Q^{[er]}_{{\cal O}_{K'}}
\to Q_{{\cal O}_{K'}}
=Q\times_{{\cal O}_K}
{\cal O}_{K'}$
by blowing up the closed
subscheme
${\rm Spec}\, {\cal O}_L
\otimes_{{\cal O}_K}
{\cal O}_{K'}/{\mathfrak m}_{K'}
^{er}$
and by removing the proper transform
of the closed fiber.
After replacing $K'$ by a separable
extension if necessary,
the normalization
$Q^{(er)}_{{\cal O}_{K'}}$
of
$Q^{[er]}_{{\cal O}_{K'}}$
has geometrically reduced closed fiber
and 
the geometric closed fiber
$Q^{(r)}_{\bar F}
=
Q^{(er)}_{{\cal O}_{K'}}
\times_{{\cal O}_{K'}}\bar F$
is independent of such $K'$.

The finite set $F^r(L)
=\pi_0(Q^{(r)}_{\bar F})$
of connected components
is independent of $Q$.
If $F(L)={\rm Mor}_K(L,\bar K)$, the mapping
$F(L)\to F^r(L)$ 
induced by the canonical morphism
${\rm Spec}\, 
({\cal O}_L\times_{{\cal O}_K}
{\cal O}_{K'})^-
\to 
Q^{(er)}_{{\cal O}_{K'}}$
from the normalization is also independent of the choice
and is a surjection.
We say that the ramification of
$L$ over $K$ is bounded by $r$
if the surjection
$F(L)\to F^r(L)$ is a bijection.
The ramification group
$G_K^r\subset G_K={\rm Gal}(\bar K/K)$
is defined to be the unique closed
normal subgroup such that
the surjection
$F(L)\to F^r(L)$
induces a bijection
$F(L)/G_K^r\to F^r(L)$.
\end{para}

\begin{para}
A logarithmic variant is defined
as follows.
Let $L$ be a finite separable
extension of $K$.
Let $m$ be an integer divisible by
the ramification index $e_{L/K}$
and $\pi$ be a prime element of $K$.
We define an extension $K_m$
to be the tamely ramified extension
$K[t]/(t^m-\pi)$ if $m$ is
invertible in $F$
and to be the fraction field of the
henselization of ${\cal O}_K[u^{\pm 1},t]
/(ut^m-\pi)$ at the prime ideal $(t)$.
Then, the finite set
$F^{mr}(L\otimes_KK_m)$
is independent of such $m$
and we define $F^r_{\log}(L)$
to be
$F^{mr}(L\otimes_KK_m)$.
We say that the log ramification of
$L$ over $K$ is bounded by $r$
if the surjection
$F(L)\to F^r_{\log}(L)$ is a bijection.
The ramification group
$G_{\log, K}^r\subset G_K$
is defined to be the unique closed
normal subgroup such that
the surjection
$F(L)\to F^r_{\log}(L)$
induces a bijection
$F(L)/G_{\log, K}^r\to F^r_{\log}(L)$

Define closed normal subgroups 
$G^{r+}_K\subset G^r_K$
and
$G^{r+}_{\log, K}
\subset G^r_{\log, K}$
to be the closures of the unions
$\bigcup_{s>r}G^s_K$
and $\bigcup_{s>r}G^s_{\log, K}$
and set $F^{r+}(L)=F(L)/G_K^{r+}$
and $F^{r+}_{\log}(L)
=F(L)/G_{K,\log}^{r+}$.
We say that the ramification 
(resp.\ the log ramification) of
$L$ over $K$ is bounded by $r+$
if the surjection
$F(L)\to F^{r+}(L)$ 
(resp.\ $F(L)\to F^{r+}_{\log}(L)$) is a bijection.
\end{para}

\begin{para}\label{s43}
We call 
the largest rational number $r$ such that
the ramification 
(resp.\ the log ramification) of
$L$ over $K$ is not bounded by $r$
the conductor 
(resp.\ the logarithmic conductor) of $L$ over $K$.
The conductor 
(resp.\ the logarithmic conductor) of $L$ over $K$
is the smallest rational number $r$ such that
the ramification 
(resp.\ the log ramification) of
$L$ over $K$ is bounded by $r+$.
The conductor $c$ and
the logarithmic conductor $c_{\log}$
satisfies the inequality $c_{\log}\leqq c$.
For an extension $K'$ of a henselian discrete valuation
field $K$ of ramification index $e$,
the conductor $c'$ and
the logarithmic conductor $c'_{\log}$
of a composition field $L'=LK'$ over $K'$
satisfy $c'_{\log}\leqq e\cdot c_{\log}$
and $c'\leqq e\cdot c$.
If $L$ is the cyclic extension defined
by an abelian character $\chi$ of $G_K$,
the Swan conductor ${\Sw}(\chi)$ 
is defined as the logarithmic conductor of $L$ over $K$.
\end{para}

\begin{lem}\label{lmlog}
Assume that the
ramification index $e_{L/K}$ is $1$.
Then, for every rational number $r>0$, the canonical surjections
$F^r_{\log}(L)\to F^r(L)$
and
$F^{r+}_{\log}(L)\to F^{r+}(L)$
are bijections.
\end{lem}

\begin{pf}
Since we may take $m=1$,
the assertion follows.
\end{pf}

\begin{para}
Let $Q\to P$ be
a quasi-finite and flat morphism
of smooth schemes over ${\cal O}_K$
and let
\begin{equation}
\begin{CD}
Q@<<< {\rm Spec}\, {\cal O}_L\\
@VVV\hspace{-15mm}\square
\hspace{10mm}@VVV\\
P@<<< {\rm Spec}\, {\cal O}_K
\end{CD}
\label{eqPQ}
\end{equation}
be a cartesian diagram.
Then, by the functoriality of
the construction of dilatations,
we obtain a finite morphism
$Q^{(r)}_{\bar F}\to
P^{(r)}_{\bar F}$
of geometric closed fibers.
Define ideals
${\mathfrak m}_{\bar K}^{r+}
\subset
{\mathfrak m}_{\bar K}^r
\subset 
{\cal O}_{\bar K}$ by
$${\mathfrak m}_{\bar K}^{r+}
=\{x\in \bar K\mid v(x)>r\}
\subset
{\mathfrak m}_{\bar K}^r
=\{x\in \bar K\mid v(x)\geqq r\}.$$
For a $k$-vector space
$V$, let
${\mathbf V}(V)$ denote
the associated covariant scheme
${\rm Spec}\, S_k^\bullet V^\vee$.
Then, since 
${\rm Spec}\, {\cal O}_K\to P$ 
is a section of
a smooth morphism,
the conormal sheaf
$N_{{\rm Spec}\, {\cal O}_K/P}$
is canonically isomorphic to
the restriction of $\Omega^1_{P/{\cal O}_K}$
and hence
the geometric closed fiber
$P^{(r)}_{\bar F}$
is canonically identified
with 
${\mathbf V}
({\rm Hom}_F(\Omega^1_{P/{\cal O}_K}
\otimes_{{\cal O}_P}F,
{\mathfrak m}_{\bar K}^r/
{\mathfrak m}_{\bar K}^{r+}))$.

The fiber
$Q^{(r)}_{\bar F}\times_{P^{(r)}_{\bar F}}
0$ of the origin
is canonically identified with
the quotient $F^{r+}(L)=F(L)/G_K^{r+}$.
The ramification of
$L$ over $K$ is bounded by $r+$
if and only if the finite morphism
$Q^{(r)}_{\bar F}\to
P^{(r)}_{\bar F}$ is \'etale.
Assume that $L$ is a Galois extension
of $K$ of Galois group $G$
and fix a morphism $L\to \bar K$.
Let $r=c$ be the conductor
of $L$ over $K$.
Then, 
the connected component
$Q^{(r)\circ }_{\bar F}$
of
$Q^{(r)}_{\bar F}$ 
containing the point corresponding to
$L\to \bar K$
is a $G^r$-torsor over
$P^{(r)}_{\bar F}$.
The conductor $r=c$ of $L$ over $K$
is characterized by the condition that
the morphism $Q^{(r)\circ }_{\bar F}\to
P^{(r)}_{\bar F}$ is finite \'etale
but is not an isomorphism.
\end{para}

\begin{lem}\label{lmmono}
Assume that ${\cal O}_L$
is generated by one element $v\in {\cal O}_L$
over ${\cal O}_K$, set
$d={\rm length}_{{\cal O}_L}
\Omega^1_{{\cal O}_L/{\cal O}_K}$ and
let ${\rm ord}_L$ be the normalized valuation.
Let $v'\neq v$ be a conjugate of $v$ such that
the valuation $s={\rm ord}_L(v'-v)$ 
is the largest.
\begin{enumerate}
\item[\rm 1.]
The rational number $r=d/[L:K]+s/e_{L/K}$ 
equals the conductor of $L$ over $K$.
\item[\rm 2.]
Let $\varphi\in {\cal O}_K[X]$ be the minimal polynomial
of $v$ and
define the left vertical arrow of
the cartesian diagram
\begin{equation}
\begin{CD}
{\mathbf A}^1_{{\cal O}_K}&=
{\rm Spec}\, {\cal O}_K[V]=Q
@<<< {\rm Spec}\, {\cal O}_L
\\
@V{U\mapsto \varphi(V)}VV&
@VVV\\
{\mathbf A}^1_{{\cal O}_K}&=
{\rm Spec}\, {\cal O}_K[U]=P
@<<< {\rm Spec}\, {\cal O}_K
\end{CD}
\label{eqmono}
\end{equation}
by $\varphi$ and the bottom horizontal arrow
by $U\mapsto 0$.
Assume that $L$ is a Galois extension
of Galois group $G$
and that $G^r=\langle\sigma\rangle$ is cyclic of order $p$.
Define isomorphisms
${\mathbf F}_p\to G^r$ by
$\sigma$ and
${\mathbf A}^1_{\bar F}
\to P^{(r)}_{\bar F}=
{\mathfrak m}_{\bar K}^r/
{\mathfrak m}_{\bar K}^{r+}$
by $\varphi'(v)(v-\sigma(v))$.
Then, there is an isomorphism
\begin{equation}
\begin{CD}
0@>>>
{\mathbf F}_p
@>>>
{\mathbf A}^1_{\bar F}
@>{T\mapsto T^p-T}>>
{\mathbf A}^1_{\bar F}
@>>>0\\
@.@VVV@VVV@VVV@.
\\
0@>>>
G^r
@>>>
Q^{(r)\circ }_{\bar F}
@>>>
P^{(r)}_{\bar F}
@>>>
0
\end{CD}
\label{eqGr}
\end{equation}
of extensions of smooth group schemes
by \'etale group schemes.
\end{enumerate}
\end{lem}

The proof is similar to the computation
in \cite[Example 3.3.3]{Sa2}.

\begin{pf}
The left vertical arrow
$Q\to P$ in (\ref{eqmono})
is finite flat.
Let $v_1,\ldots,v_n\in L$ be the conjugates
of $v$. We fix a numbering so that 
$v_n=v, v_{n-1}=\sigma(v)$
and ${\rm ord}_L(v_i-v_n)$ is increasing.
Setting $X-v_n=(v_{n-1}-v_n)T$,
we have
$\varphi(X)=\prod_{i=1}^n(X-v_i)=
\prod_{i=1}^n(v_n-v_i+(v_{n-1}-v_n)T)$.
By the assumption that
$G^r$ is cyclic of order $p$,
we have ${\rm ord}_L(v_i-v_n)
<{\rm ord}_L(v_{n-1}-v_n)$ for
$i\leqq n-p$
and we may assume that
$(v_{n-i}-v_n)/(v_{n-1}-v_n)\equiv i
\bmod {\mathfrak m}_L$
for $i=0,\ldots,p-1$.
Hence, we have
\begin{equation}
\varphi(X)\equiv
\prod_{i=1}^{n-1}
(v_n-v_i)
\cdot
\prod_{i=1}^{p-1}(1+iT)
\cdot (v_{n-1}-v_n)T
=
\varphi'(v)(v-\sigma(v))(T^p-T)
\label{eqphi}
\end{equation}
$\bmod\ {\mathfrak m}_L^{r+1}$.
Thus the assertion 1 follows from
the characterization of the conductor
at the end of 4.5.
The assertion 2 also follows from (\ref{eqphi}).
\end{pf}

\begin{lem}\label{lmN}
Let $P_1,P_2,Q_1,Q_2$ be smooth schemes
over ${\cal O}_K$
and 
\begin{equation}
\begin{CD}
Q_i@<<< {\rm Spec}\, {\cal O}_L\\
@V{f_i}VV \hspace{-12mm}\square
\hspace{12mm}@VVV\\
P_i@<<< {\rm Spec}\, {\cal O}_K
\end{CD}
\label{eqPQi}
\end{equation}
for $i=1,2$ be cartesian diagrams
of schemes over ${\cal O}_K$
such that the vertical arrows are
quasi-finite and flat.
Let 
\begin{equation}
\begin{CD}
Q_1\times_{{\cal O}_K}F
@<<<
Q_2\times_{{\cal O}_K}F
@<<< {\rm Spec}\, {\cal O}_L
\otimes_{{\cal O}_K}F\\
@V{\bar f_1}VV@V{\bar f_2}VV @VVV\\
P_1\times_{{\cal O}_K}F
@<<<
P_2\times_{{\cal O}_K}F
@<<< {\rm Spec}\, F
\end{CD}
\label{eqQ1Q2}
\end{equation}
be a commutative diagram where the right square is induced by
{\rm (\ref{eqPQi})}.
Then for a rational number $r>0$,
the diagram {\rm (\ref{eqQ1Q2})}
induces a commutative diagram
\begin{equation}
\begin{CD}
Q^{(r)}_{1, \bar F}
@<<<
Q^{(r)}_{2, \bar F}
\\
@V{\bar f_1^{(r)}}VV@VV{\bar f_2^{(r)}}V\\
P^{(r)}_{1, \bar F}
@<<<
P^{(r)}_{2, \bar F}.
\end{CD}
\label{eqQ1Q2b}
\end{equation}
\end{lem}

\begin{pf}
For $i=1,2$,
we consider the unions
$Q_i'=(Q_i\times_{{\cal O}_K}F)
\cup {\rm Spec}\, {\cal O}_L
\subset Q_i$
as reduced closed subschemes.
Then by the commutative diagram
(\ref{eqQ1Q2}), the morphism 
$Q_1\times_{{\cal O}_K}F
\to
Q_2\times_{{\cal O}_K}F$
and the identity of
${\rm Spec}\, {\cal O}_L$
define a morphism
$Q'_1\gets Q'_2$.
Since $Q_1$ is smooth over ${\cal O}_K$,
after replacing $Q_2$ by an \'etale neighborhood
of ${\rm Spec}\, {\cal O}_L$ if necessary,
we may lift 
$Q'_1\gets Q'_2$ to a morphism
$Q_1\gets Q_2$ over ${\cal O}_K$.

The morphism
$Q_1\gets Q_2$
induces a morphism of conormal modules
$N_{{\rm Spec}\, {\cal O}_L/Q_1}
\to 
N_{{\rm Spec}\, {\cal O}_L/Q_2}$
and defines a commutative diagram
\begin{equation}
\begin{CD}
Q^{(r)}_{1, \bar F}
@<<<
Q^{(r)}_{2, \bar F}
\\
@VVV@VVV\\
{\mathbf V}(
{\rm Hom}_F(N_{{\rm Spec}\, {\cal O}_L/Q_1}
\otimes_{{\cal O}_L}\bar F,
{\mathfrak m}_{\bar K}^r/
{\mathfrak m}_{\bar K}^{r+})
)
@<<<
{\mathbf V}(
{\rm Hom}_F(N_{{\rm Spec}\, {\cal O}_L/Q_2}
\otimes_{{\cal O}_L}\bar F,
{\mathfrak m}_{\bar K}^r/
{\mathfrak m}_{\bar K}^{r+})
).
\end{CD}
\label{eqQ1Q2N}
\end{equation}
By the cartesian diagram (\ref{eqPQi}),
the conormal modules
$N_{{\rm Spec}\, {\cal O}_L/Q_i}$
are the tensor products
$N_{{\rm Spec}\, {\cal O}_K/P_i}
\otimes_{{\cal O}_K}{\cal O}_L$
for $i=1,2$.
Hence by the commutative diagram
(\ref{eqQ1Q2}),
we may replace
${\mathbf V}(
{\rm Hom}_F(N_{{\rm Spec}\, {\cal O}_L/Q_i}
\otimes_{{\cal O}_L}\bar F,
{\mathfrak m}_{\bar K}^r/
{\mathfrak m}_{\bar K}^{r+})
)
$
by $P^{(r)}_{i, \bar F}$
to get (\ref{eqQ1Q2b}).
\end{pf}

By slightly enlarging the terminology,
we say that the henselization of
a local ring of a scheme of finite
type is essentially of finite type.

\begin{lem}[cf. {\cite[Lemma 4.4, 4.5]{AS2}}]
\label{lmPQ}
Let ${\cal O}_K$ be a
henselian discrete valuation ring
essentially of finite type and flat
over $W=W(k)$
for a perfect field $k$ of characteristic $p>0$.
\begin{enumerate}
\item[\rm 1.]
There exist
a smooth scheme $P_0$
over $W$,
a divisor $D_0\subset P_0$
smooth over $W$,
a divisor $X_0\subset P_0$
flat over over $W$
meeting $D_0$ transversely 
and an isomorphism
${\cal O}^h_{X_0,\xi}\to {\cal O}_K$ over $W$
from the henselization 
of the local ring at
a generic point $\xi$ of the intersection
$X_0\cap D_0$.
\item[\rm 2.]
Let $L$ be a finite separable extension of
$K$ of ramification index $e$.
Let $P_0,D_0,X_0$ be as in $1$.
Further let $Q_0$ be
a smooth scheme over $W$,
$E_0\subset Q_0$ a smooth divisor over $W$,
$Y_0\subset Q_0$ a divisor flat over $W$
meeting $E_0$ transversely 
and ${\cal O}^h_{Y_0,\eta}\to {\cal O}_L$
an isomorphism  over $W$.
Then, after replacing $X_0$ and $Y_0$
by \'etale neighborhoods of
${\rm Spec}\,{\cal O}_K$ and
${\rm Spec}\,{\cal O}_L$,
there exists a cartesian diagram
\begin{equation}
\begin{CD}
e\cdot E_0@>>> Q_0@<<< Y_0
@<<< {\rm Spec}\,{\cal O}_L\\
@VVV@VfVV\hspace{-10mm}\square
\hspace{6mm} @VVV
\hspace{-15mm}\square
\hspace{11mm} @VVV\\
D_0@>>> P_0@<<< X_0@<<<
{\rm Spec}\,{\cal O}_K
\end{CD}
\label{eqPQ0}
\end{equation}
such that the vertical arrows are finite flat.
\end{enumerate}
\end{lem}

\begin{pf}
1.
Let $u_1,\ldots,u_n
\in {\cal O}_K$ be liftings of
a transcendental
basis $\bar u_1,\ldots,\bar u_n
\in F$ over $k$
such that $F$ is a finite
separable extension of $k(\bar u_1,\ldots,\bar u_n)$
and $\pi$ be a prime element of $K$.
If we set ${\mathbf A}^{n+1}_{W}={\rm Spec}\, W
[u_1,\ldots,u_n,t]$, the morphism
${\cal O}_K\to {\mathbf A}^{n+1}_{W}$ defined by
$u_1,\ldots,u_n,\pi\in  {\cal O}_K$ is
formally unramified.
Hence, there exists
an \'etale neighborhood
$P_0\to {\mathbf A}^{n+1}_{W}$
of the image $\xi$ of
the closed point
of ${\rm Spec}\, {\cal O}_K$,
a regular divisor
$X_0\subset P_0$
and an isomorphism
${\cal O}^h_{X_0,\xi}\to {\cal O}_K$.
It suffices to define $D_0$ by $t$.

2.
Take a function on $P_0$
defining $D_0$ and take
an \'etale morphism
$P_0\to  {\mathbf A}^{n+1}_{W}={\rm Spec}\, W
[u_1,\ldots,u_n,t]$
such that $D_0\subset P_0$
is defined by $t$.
Let $\pi\in {\cal O}_K$
be the image of $t$.

Let $s$ be a function on $Q_0$
defining $E_0$
and let $\pi'\in {\cal O}_L$
be the image of $s$.
Define $v\in {\cal O}_L^\times$
by $\pi=v\pi^{\prime e}$
and lift it to a unit $\tilde v$ on $Q_0$.
We define a morphism
$Q_0\to {\mathbf A}^{n+1}_{W}$ satisfying $t\mapsto 
\tilde v s^e$ and lifting the composition
${\rm Spec}\, {\cal O}_L
\to {\rm Spec}\, {\cal O}_K\to
{\mathbf A}^{n+1}_{W}$.
By replacing $Q_0$ by an \'etale
neighborhood,
we may lift $Q_0\to 
{\mathbf A}^{n+1}_{W}$
to $f\colon Q_0\to P_0$
satisfying $f^*D_0=e\cdot E_0$.

We show that the middle and the right squares are
cartesian after replacing
$Q_0$ and $P_0$ by \'etale neighborhoods.
Since the residue fields $F$
and $E$ of $K$ and $L$
are the function field
of the closed fibers
$D_{0,k}$ and $E_{0,k}$,
we may assume
$E_0\to D_0$
and hence
$Q_0\to P_0$
are quasi-finite and hence
flat.
Further, we may assume that
$Q_0\to P_0$ is finite flat
and the right square is cartesian.
Then, the morphism
$Q_0\to P_0$ is of degree $[L:K]$
and hence the middle square is cartesian.
\end{pf}

\begin{para}
Assume that ${\cal O}_K$
is essentially of finite type and flat
over $W=W(k)$
for a perfect field $k$ of characteristic $p>0$
and let the notation be as in
Lemma \ref{lmPQ}.2.
We define a dilatation
$$P^\sim=
(P_0\times_{W}{\cal O}_K)^\sim
\to P=P_0\times_{W}{\cal O}_K$$
by blowing-up
$D_0\times_WF
=
(D_0\times_W{\cal O}_K)
\cap
(P_0\times_WF)
\subset
P_0\times_W{\cal O}_K$
and by removing the proper
transforms of
$D_0\times_W{\cal O}_K$
and of $P_0\times_WF$.
We consider a cartesian diagram
\begin{equation}
\begin{CD}
Q_0@<<< Q@<<< 
Q^\sim @<<< {\rm Spec}\, {\cal O}_L\\
@VVV
\hspace{-10mm}\square\hspace{5mm}
@VVV
\hspace{-10mm}\square\hspace{5mm}
@VVV 
\hspace{-15mm}\square\hspace{10mm}
@VVV\\
P_0@<<< P@<<< 
P^\sim @<<< {\rm Spec}\, {\cal O}_K.
\end{CD}
\label{eqPQQ}
\end{equation}

We consider ${\cal O}_K$
as a log scheme
with the log structure defined
by the closed point.
With respect to the log structure
of $Q^\sim$ defined by the pull-back
of $E_0$,
the log scheme $Q^\sim$
is log smooth over ${\cal O}_K$.
Let $K'$ be a finite separable extension
such that
the ramification index $e_{K'/K}$
is divisible by $e=e_{L/K}$.
Then, the log product
$Q^\sim_{{\cal O}_{K'}}=
Q^\sim\times^{\log}_{{\cal O}_K}
{\cal O}_{K'}$ is classically smooth
over ${\cal O}_{K'}$
and we have a cartesian diagram
\begin{equation}
\begin{CD}
Q^\sim_{{\cal O}_{K'}}
@<<< {\rm Spec}\, 
({\cal O}_L\times^{\log}_{{\cal O}_K}
{\cal O}_{K'})\\
@VVV 
\hspace{-25mm}\square\hspace{15mm}
@VVV\\
P^\sim_{{\cal O}_{K'}}
@<<< {\rm Spec}\, {\cal O}_{K'}.
\end{CD}
\label{eqPQK'}
\end{equation}

By the cartesian diagram
(\ref{eqPQQ}),
the conormal modules
$N_{{\rm Spec}\, {\cal O}_L/Q}$
and 
$N_{{\rm Spec}\, {\cal O}_L/Q^\sim}$
are the pull-backs of
$N_{{\rm Spec}\, {\cal O}_K/P}
=\Omega^1_{{P_0}/W}\otimes_{{\cal O}_{P_0}}
{\cal O}_K$
and 
$N_{{\rm Spec}\, {\cal O}_K/P^\sim}
=\Omega^1_{{P_0}/W}(\log D_0)
\otimes_{{\cal O}_{P_0}}
{\cal O}_K$.
We have an exact sequence
$0\to N_{D_0/P_0}
\to 
\Omega^1_{{P_0}/W}\otimes_{{\cal O}_{P_0}}
{\cal O}_{D_0}
\to
\Omega^1_{{D_0}/W}\to0$
and a commutative diagram
$$\begin{CD}
\Omega^1_{{D_0}/W}\otimes_{{\cal O}_{D_0}}
F@>>> \Omega^1_F\\
@VVV@VVV\\
\Omega^1_{{P_0}/W}(\log D_0)
\otimes_{{\cal O}_{D_0}}
F@>>> \Omega^1_F(\log)
\end{CD}$$
where the horizontal arrows are canonical isomorphisms.
Hence the diagram
(\ref{eqPQQ}) and (\ref{eqPQK'})
define a commutative diagram
\begin{equation}
\begin{CD}
Q^{(r)}_{\bar F}@<<< Q^{\sim (r)}_{\bar F} 
\\
@VVV @VVV\\
P^{(r)}_{\bar F}@<<< P^{\sim (r)}_{\bar F}
\end{CD}
\label{eqPQr}
\end{equation}
of the reduced closed fibers of dilatations.
The bottom arrow
is the linear mapping
\begin{align}
P^{(r)}_{\bar F}
=
{\mathbf V}
({\rm Hom}_F
(\Omega^1_{{P_0}/W}
\otimes_{{\cal O}_{D_0}}
&F,
{\mathfrak m}_{\bar K}^r/
{\mathfrak m}_{\bar K}^{r+}))\\
&\gets 
P^{\sim (r)}_{\bar F}
=
{\mathbf V}
({\rm Hom}_F
(\Omega^1_F(\log),
{\mathfrak m}_{\bar K}^r/
{\mathfrak m}_{\bar K}^{r+}))
\nonumber
\end{align}
of $\bar F$-vector spaces
and its image is
${\mathbf V}
({\rm Hom}_F
(\Omega^1_F,
{\mathfrak m}_{\bar K}^r/
{\mathfrak m}_{\bar K}^{r+}))
.$

Assume that $L$ is a Galois
extension of $K$ of Galois group $G$
and let $r=c_{\log}>0$ be the logarithmic conductor of $L$
 over $K$.
We fix a morphism $L\to \bar K$ over $K$.
By \cite[Theorem 2]{Sa},
the right vertical arrow of (\ref{eqPQr})
restricted to the connected
component $Q^{\sim (r)\circ}_{\bar F}
\subset Q^{\sim (r)}_{\bar F}$
containing the point corresponding to
$L\to \bar K$
defines an extension 
\begin{equation}
\begin{CD}
0@>>> G_{\log}^r
@>>>
Q^{\sim (r)\circ}_{\bar F}
@>>>
P^{\sim (r)}_{\bar F}
=
{\mathbf V}
({\rm Hom}_F
(\Omega^1_F(\log),
{\mathfrak m}_{\bar K}^r/
{\mathfrak m}_{\bar K}^{r+}))
@>>> 0
\end{CD}
\label{eqExt}
\end{equation}
of an $\bar F$-vector space
by an ${\mathbf F}_p$-vector space.
By \cite[Proposition 1.20]{WC},
the class of the extension (\ref{eqExt}) defines an element 
\begin{equation}
\omega\in G_{\log}^r\otimes_{{\mathbf F}_p} 
\Omega^1_F(\log)\otimes_F
{\mathfrak m}_{\bar K}^{-r}/
{\mathfrak m}_{\bar K}^{-r+}.
\label{eqom}
\end{equation}
This is independent of the choice of
the diagram (\ref{eqPQQ}) and is
called the refined logarithmic conductor of $L$ over $K$.
If $G$ is cyclic and
$\chi\colon G\to {\mathbb C}^\times$ is an injective
abelian character of $G$ and an injection
${\mathbb Z}/p\to {\mathbb C}^\times$ is fixed,
the image of $\omega$ in
$\Omega^1_F(\log)\otimes_F
{\mathfrak m}_{\bar K}^{-r}/
{\mathfrak m}_{\bar K}^{-r+}$
is the refined Swan conductor
\begin{equation}
{\rsw}(\chi)\in 
\Omega^1_F(\log)\otimes_F
{\mathfrak m}_{\bar K}^{-r}/
{\mathfrak m}_{\bar K}^{-r+}.
\label{eqrsw}
\end{equation}
\end{para}

\begin{lem}\label{lmlogrsw}
Let $L$ be a finite Galois extension of $K$
of Galois group $G$.
Let $r$ be the logarithmic conductor
of $L$ over $K$
and 
$\omega\in G^r_{\log}\otimes_{{\mathbf F}_p} 
\Omega^1_F(\log)\otimes_F
{\mathfrak m}_{\bar K}^{-r}/
{\mathfrak m}_{\bar K}^{-r+}$ be
the refined logarithmic conductor.
\begin{enumerate}
\item[\rm 1.]
If the conductor of $L$ over $K$
is the same as the logarithmic conductor $r$ of $L$ over $K$,
then the refined logarithmic conductor $\omega$
is in the image of
$G^r_{\log}\otimes_{{\mathbf F}_p}
\Omega^1_F\otimes_F
{\mathfrak m}_{\bar K}^ {-r}/
{\mathfrak m}_{\bar K}^{-r+}$.
\item[\rm 2.]
Let $K'$ be an extension of
henselian valuation fields of $K$
of ramification index $e$
and of residue field $F'$.
Assume that the image
$\omega'\in G^r_{\log}\otimes_{{\mathbf F}_p} 
\Omega^1_{F'}(\log)\otimes_{F'}
{\mathfrak m}_{\bar K'}^{-r}/
{\mathfrak m}_{\bar K'}^{-r+}$
of the refined logarithmic conductor $\omega$
of $L$ over $K$
is non-trivial.

Then, 
the logarithmic conductor $r'$ of 
a composition field $L'$ over $K'$
equals $er$
and $\omega'$
is the image of
the refined logarithmic conductor 
of $L'$ over $K'$
by the morphism induced by
the injection 
${\rm Gal}(L'/K')_{\log}^{r'}
\to G^r_{\log}$.
\end{enumerate}
\end{lem}

\begin{pf}
1.
We may assume that the residue field $F$
of $K$ is of finite type
over a perfect subfield $k$.
Then, the assertion follows from the
commutative diagram
(\ref{eqPQr}).

2.
By the functoriality of construction,
the logarithmic ramification of
$L'$ over $K'$ is bounded by $er+$.
We may assume that the residue fields $F$
and $F'$ 
of $K$ and $K'$ are of finite type
over perfect subfields $k\subset k'$.
Then, further by the functoriality of construction,
we obtain a morphism
$Q^{\prime (er)}_{\bar F'}\to
Q^{(r)}_{\bar F}$ compatible
with the injection $G'={\rm Gal}(L'/K')\to G$
and a commutative diagram
\begin{equation}
\begin{CD}
0@>>> G^{\prime er}_{\log}
@>>> Q^{\prime (er)\circ}_{\bar F'}
@>>>
{\mathbf V}
({\rm Hom}_{F'}
(\Omega^1_{F'}(\log),
{\mathfrak m}_{\bar K'}^{er}/
{\mathfrak m}_{\bar K'}^{er+}))
@>>>0\\
@.@VVV@VVV@VVV@.\\
0@>>> G^r_{\log}
@>>> Q^{(r)\circ}_{\bar F}
@>>>
{\mathbf V}
({\rm Hom}_F
(\Omega^1_F(\log),
{\mathfrak m}_{\bar K}^r/
{\mathfrak m}_{\bar K}^{r+}))
@>>>0
\end{CD}
\label{eqExtb}
\end{equation}
of extensions.
Since $\omega'$ is the extension class
of the pull-back of the lower
line by the right vertical arrow,
the assumption $\omega'\neq 0$
means that the pull-back is non-trivial
and $G^{\prime er}_{\log}\neq 0$.
Hence $er$ is the logarithmic conductor of $L'$ over $K'$.
The last assertion also follows from the diagram
(\ref{eqExtb}).
\end{pf}

\begin{prop}\label{prmain}
Assume that the residue field
$F$ of $K$ is a function field of
one variable over a perfect subfield $k$
of characteristic $p>0$
and that the characteristic of $K$ is $0$.
Let $u\in {\cal O}_K$ be a lifting
of an element $\bar u\in F$
such that $F$ is a finite separable extension
of $k(\bar u)$.

Let $L$ be a finite Galois extension of $K$
of Galois group $G$.
Assume that
the ramification index is $1$ 
and that the residue field $E$ is a purely inseparable
extension of $F$.
Let $v\in {\cal O}_L$ be a lifting
of a generator $\bar v\in E=F(\bar v)$
and let $\varphi\in {\cal O}_K[T]$
be the minimal polynomial of $v$.
Assume that $\varphi\equiv
T^q-\bar u\bmod {\mathfrak m}_K$.

Let $r$ be
the logarithmic conductor of $L$ over $K$.
Assume that $G^r$ is cyclic of order $p$
and identify $G^r=\langle \sigma\rangle$
with ${\mathbf F}_p$ by fixing a
generater $\sigma$.
Then, 
$r={\rm ord}_L\varphi'(v)(v-\sigma(v))$
and the refined logarithmic conductor
of $L$ over $K$ is
$$\dfrac {d\bar u}
{\varphi'(v)(v-\sigma(v))}
\in \Omega^1_F
\otimes_F {\mathfrak m}^{-r}_{\bar K}
/{\mathfrak m}^{-r+}_{\bar K}.$$
\end{prop}

Almost the same result 
as Proposition \ref{prmain} is proved in 
\cite[Theorem 5.9]{Hu} in a similar way.
Although we assume that $K$ is of mixed characteristic in
Proposition \ref{prmain},
the same assertion is proved more easily in the equal characteristic case.

\begin{pf}
By Lemma \ref{lmlog},
the logarithmic conductor equals
the conductor.
The equality 
$r={\rm ord}_L\varphi'(v)(v-\sigma(v))$
follows then from Lemma \ref{lmmono}.
We use the notation in
Lemma \ref{lmPQ}.2.
Since $D_0\subset P_0$ is smooth over $W$,
there exists a smooth morphism
$P_0\to {\mathbf A}^1_W$
such that $D_0$ is the pull-back of
the $0$-section ${\rm Spec}\, W\to {\mathbf A}^1_W$.
By the assumption that $e=1$,
the divisor $E_0\subset Q_0$ is also the pull-back of
the $0$-section ${\rm Spec}\, W\to {\mathbf A}^1_W$.

Let $P_1=P_0\times_{{\mathbf A}^1_W}{\cal O}_K$
and $Q_1=Q_0\times_{{\mathbf A}^1_W}{\cal O}_K$
be the fiber products with respect to
the composition ${\cal O}_K\to P_0\to
{\mathbf A}^1_W$.
Then $P_1$ and $Q_1$ are
also smooth over ${\cal O}_K$ and we have a cartesian diagram
$$\begin{CD}
Q@<<< Q_1@<<< {\rm Spec}\, {\cal O}_L\\
@VVV@VVV@VVV\\
P@<<< P_1@<<< {\rm Spec}\, {\cal O}_K.
\end{CD}$$
By this and Lemma \ref{lmlogrsw},
the refined logarithmic Swan conductor
is the image of the class of
the extension
\begin{equation}
0\to G^r\to
Q^{(r)\circ}_{1, \bar F}
\to P^{(r)}_{1, \bar F}
={\mathbf V}(\Omega^1_F
\otimes_F {\mathfrak m}_{\bar K}^{-r}
/{\mathfrak m}_{\bar K}^{-r+})\to 0
\label{eqExt1}
\end{equation}
defined as the restriction of the $G^r$-torsor
$Q^{(r)\circ }_{\bar F}$ over
$P^{(r)}_{\bar F}$.

We compute the extension (\ref{eqExt1})
by comparing it using Lemma \ref{lmN}
with that defined by the cartesian diagram (\ref{eqmono})
in Lemma \ref{lmmono}.
Since the diagram
$$\begin{CD}
Q_0\times_{{\mathbf A}^1_W}k
@<<< {\rm Spec}\, E\\
@VVV@VVV\\
P_0\times_{{\mathbf A}^1_W}k
@<<< {\rm Spec}\, F
\end{CD}$$
is commutative
and the horizontal arrows
are the canonical morphisms of the generic points,
we have a commutative diagram
$$\begin{CD}
{\rm Spec}\, k[V]=&
{\mathbf A}^1_k
@<<<Q_0\times_{{\mathbf A}^1_W}k
\\
&@V{U\mapsto V^q}VV@VVV\\
{\rm Spec}\, k[U]=&
{\mathbf A}^1_k
@<<<
P_0\times_{{\mathbf A}^1_W}k
\end{CD}$$
where the horizontal arrows are \'etale.
This induces a commutative diagram
$$\begin{CD}
{\rm Spec}\, F[V]=&
{\mathbf A}^1_F
@<<<Q_1\times_{{\cal O}_K}F
@<<< {\rm Spec}\, E
\\
&@V{U\mapsto V^q}VV@VVV@VVV\\
{\rm Spec}\, F[U]=&
{\mathbf A}^1_F
@<<<P_1\times_{{\cal O}_K}F
@<<< {\rm Spec}\, F.
\end{CD}$$
By comparing this with
the cartesian diagram
$$\begin{CD}
{\rm Spec}\, F[V]=&
{\mathbf A}^1_F
@>>>{\mathbf A}^1_{{\cal O}_K}
@<<< {\rm Spec}\, {\cal O}_L
\\
&@V{U\mapsto V^q}VV@V{U\mapsto \varphi(V)}VV
@VVV\\
{\rm Spec}\, F[U]=&
{\mathbf A}^1_F
@>>>{\mathbf A}^1_{{\cal O}_K}
@<<< {\rm Spec}\, {\cal O}_K
\end{CD}$$
obtained by (\ref{eqmono}),
we obtain a commutative diagram 
$$\begin{CD}
{\mathbf A}^1_{\bar F}@>>>Q^{(r)}_{1, \bar F}\\
@VVV@VVV\\
{\mathbf A}^1_{\bar F}@>>>
P^{(r)}_{1, \bar F}
\end{CD}$$
by Lemma \ref{lmN}.
Since the left vertical arrow is as in Lemma \ref{lmmono}.2
and the bottom isomorphism
is defined by $U-u$,
the assertion follows.
\end{pf}

\section{Coincidence of Swan conductors and of refined Swan conductors}\label{s4}

We prove properties of $\Sw^{\ab}$ and $\rsw^{\ab}$
similar to Lemma \ref{lmlogrsw}.2 and Proposition \ref{prmain}.

\begin{prop}\label{SwK'K}
Let $K$ be a complete discrete valuation field
and let $\chi\colon {\rm Gal}(L/K)\to {\mathbb C}^\times$
be a character for a finite abelian extension $L$ of $K$.
Let $K'$ over $K$ be an extension of
complete discrete valuation fields of ramification index $e=e(K'/K)$
and let $\chi'\colon {\rm Gal}(LK'/K')\to {\mathbb C}^\times$
be the composition of $\chi$ with
the canonical morphism ${\rm Gal}(LK'/K')\to{\rm Gal}(L/K)$.
\begin{enumerate}
\item[\rm 1.]
We have
$\Sw^{\ab}_{K'}\chi'\leq e\cdot \Sw^{\ab}_K\chi$.
\item[\rm 2.]
Assume that
$r=\Sw^{\ab}_K\chi\geq 1$.
Then, the following conditions are equivalent:
\begin{enumerate}
\item[\rm (1)]
We have
$\Sw^{\ab}_{K'}\chi'= e\cdot \Sw^{\ab}_K\chi$.
\item[\rm (2)]
The image of $\rsw^{\ab}\chi$ by the canonical morphism
$$\mathfrak m_K^{-r}/\mathfrak m_K^{-r+1}
\otimes_{{\cal O}_K} \Omega^1_{{\cal O}_K}(\log)
\to
\mathfrak m_{K'}^{-er}/\mathfrak m_{K'}^{-er+1}
\otimes_{{\cal O}_{K'}} \Omega^1_{{\cal O}_{K'}}(\log)$$
is non-zero.
\end{enumerate}
If the equivalent conditions hold,
$\rsw^{\ab}\chi'$ equals the image of $\rsw^{\ab}\chi$.
\end{enumerate}
\end{prop}

\begin{proof}
1.
Let $r=\Sw^{\ab}_K\chi$ and $\pi$ be a prime element of $K$.
By \cite[Proposition (6.3)]{KKs},
we have $\{\chi',1+\pi^r{\mathfrak m}_{K'}\}=0$
and the assertion follows.

2.
The condition (2) is equivalent to that
$\{\chi',1+\pi^rT\}\neq 0$
in ${\rm Br}(L')$ where
$L'$ is the field of fractions of
the henselization of ${\cal O}_{K'}[T]_{(\pi')}$
for a prime element $\pi'$ of $K'$.
Hence, this is equivalent to (1).
Further, since the equality 
$\{\chi,1+\pi^rT\}=\lambda_\pi(T\alpha,T\beta)$
is compatible with base change,
$\rsw^{\ab}\chi'$ equals the image of $\rsw^{\ab}\chi$.
\end{proof}

\begin{prop}\label{Swe=1}
Let $K$ be a complete discrete valuation field such
that the residue field $F$ is of
characteristic $p>0$ and $[F:F^p]=p$.
Let 
$\chi\colon {\rm Gal}(L/K)\to {\mathbb C}^\times$
be a faithful character for a cyclic extension $L$ of $K$
of degree $q=p^e$
such that $e(L/K)=1$
and that the residue field $E$ of $L$
is a purely inseparable extension of $F$.

Let $v\in {\cal O}_L$ be a lifting
of a generator $\bar v\in E=F(\bar v)$
and let $\varphi\in {\cal O}_K[T]$
be the minimal polynomial of $v$.
Let $\sigma\in  {\rm Gal}(L/K)$ be an element of
order $p$ and set $\bar u=\bar v^q\in F$ and
$r={\rm ord}_L\varphi'(v)(v-\sigma(v))$.

Then, we have $\Sw^{\ab}\chi=r$ and
$$\rsw^{\ab}\chi=\dfrac {d\bar u}
{\varphi'(v)(v-\sigma(v))}
\in {\mathfrak m}^{-r}_K
/{\mathfrak m}^{-r+1}_K
\otimes_F \Omega^1_F
\subset {\mathfrak m}^{-r}_K
/{\mathfrak m}^{-r+1}_K
\otimes_{{\cal O}_K} \Omega^1_{{\cal O}_K}(\log).
$$
\end{prop}

\begin{proof}
The assertion follows from \cite[Proposition (6.3)]{KKs}
and \cite[Theorem (3.6)]{Ksd}.
\end{proof}

\begin{para}
We prove Theorems \ref{th01} and \ref{th02}. 

Let $K$ be a complete discrete valuation
field with residue field of characteristic $p>0$.
We may assume that $K$ is of characteristic $0$.
Let $L$ be a finite cyclic extension of $K$
and $\chi\colon {\rm Gal}(L/K)\to {\mathbb C}^\times$ be a faithful character.
We may assume that $L$ is not
tamely ramified.
We may further assume that
the residue field $F$ of $K$
is of finite type over a perfect subfield $k$,
by a standard limit argument.

By Theorem \ref{th2},
there exists an extension $K'$ over $K$
of complete discrete valuation
fields satisfying the conditions (i), (ii)' and (iii) in Theorem \ref{th2}.
Let $e=e(K'/K)$ be the ramification index
and 
$\chi'\colon {\rm Gal}(LK'/K')\to {\mathbb C}^\times$ be the character
induced by $\chi$.
Then, by the condition (iii),
the images of ${\rm rsw}(\chi)$
and ${\rm rsw}^{\rm ab}(\chi)$ are
non-zero. 

Hence
by Proposition \ref{SwK'K}
and Lemma \ref{lmlogrsw}.2, 
we have 
${\rm Sw}(\chi')=e\cdot {\rm Sw}(\chi)$
and 
${\rm Sw}^{\rm ab}(\chi')=e\cdot {\rm Sw}^{\rm ab}(\chi)$.
Further ${\rm rsw}(\chi')$
and ${\rm rsw}^{\rm ab}(\chi')$
are the images of
${\rm rsw}(\chi)$
and ${\rm rsw}^{\rm ab}(\chi)$
respectively.
Thus the equality
${\rm Sw}(\chi)={\rm Sw}^{\rm ab}(\chi)$
is equivalent to
${\rm Sw}(\chi')={\rm Sw}^{\rm ab}(\chi')$.
Further by the condition (iii) in Theorem \ref{th2},
the equality
${\rm rsw}(\chi')=
{\rm rsw}^{\rm ab}(\chi')$
is equivalent to
${\rm rsw}(\chi)={\rm rsw}^{\rm ab}(\chi)$.
Thus, we may assume that the conditions (i) and (ii)' in Theorem \ref{th2}
are satisfied.
In this case, the assertion follows from Propositions \ref{Swe=1} and  \ref{prmain}.
\end{para}

\bibliographymark{References}

\providecommand{\bysame}{\leavevmode\hbox to3em{\hrulefill}\thinspace}
\providecommand{\arXiv}[2][]{\href{https://arxiv.org/abs/#2}{arXiv:#1#2}}
\providecommand{\MR}{\relax\ifhmode\unskip\space\fi MR }
\providecommand{\MRhref}[2]{%
  \href{http://www.ams.org/mathscinet-getitem?mr=#1}{#2}
}
\providecommand{\href}[2]{#2}

\end{document}